\documentclass[12pt,a4paper]{amsart}
\usepackage[a4paper,top=2.5cm, left=2.5cm, right=2.5cm, bottom=2.5cm]{geometry}
\usepackage{lscape}
\usepackage{amsfonts,amsthm,amsmath,amssymb,latexsym}
\usepackage[T1]{fontenc}
\usepackage{mathpazo}

\usepackage{hyperref}
\usepackage{url}
\newtheorem{thm}{Theorem}
\newtheorem{cor}[thm]{Corollary}
\newtheorem{lem}[thm]{Lemma}

\theoremstyle{definition}

\theoremstyle{remark}

\newcommand{\dd}{\mathinner{\ldotp\ldotp}}

\begin{document}

\title{The total run length of a word}

\author{Amy Glen}%
\address{Department of Mathematics and Statistics \newline
\indent Murdoch University
\newline
\indent  90 South Street
\newline
\indent Murdoch, Western Australia 6150\newline\indent Australia}%
\email{\href{mailto:amy.glen@gmail.com}{amy.glen@gmail.com}}

\author{Jamie Simpson}%
\address{Department of Mathematics and Statistics \newline
\indent Curtin University of Technology\newline \indent GPO Box
U1987 \newline
\indent Perth, Western Australia 6845\newline\indent Australia}%
\email{\href{mailto:simpson@maths.curtin.edu.au}{simpson@maths.curtin.edu.au}}

\begin{abstract} A \emph{run} in a word is a periodic factor whose length is
at least twice its period and which cannot be extended to the left
or right (by a letter) to a factor with greater period.  In recent years a great
deal of work has been done on estimating the maximum number of runs
that can occur in a word of length $n$. A number of associated
problems have also been investigated.  In this paper we consider a
new variation on the theme.  We say that the \emph{total run length} (TRL)
of a word is the sum of the lengths of the runs in the word and that
$\tau(n)$ is the maximum TRL over all words of length $n$.  We
show that $n^2/8 < \tau(n)
<47n^2/72+2n$ for all $n$. We also give a formula for the average total run
length of words of length $n$ over an alphabet of size $\alpha$, and
some other results.
\end{abstract}

\maketitle

\section{Introduction}
 We use notation for combinatorics on words. A word of $n$ elements is $x=x[1 \dd n]$, with $x[i]$ being
the $i$th element and $x[i\dd j]$ the \emph{factor} of elements from
position $i$ to position $j$.  If $i=1$ then the factor is a
\emph{prefix} and if $j=n$ it is a \emph{suffix}. The letters in $x$
come from some \emph{alphabet} $A$.  The \emph{length} of $x$,
written $|x|$, is the number of letters $x$ contains and the number
of occurrences of a letter $a$ in $x$ is denoted by $|x|_a$. Two or
more adjacent identical factors form a so-called \emph{power}. A
word which is not a power is said to be \emph{primitive}. A word $x$
or factor $x$ is \emph{periodic} with period $p$ if $x[i]=x[i+p]$
for all $i$ such that $x[i]$ and $x[i+p]$ are in $x$. A periodic
word with least period $p$ and length $n$ is said to have
\emph{exponent} $n/p$. For example, the word $ababa$ has exponent
$5/2$ and can be written as $(ab)^{5/2}$.
 If $x=x[1\dd n]$ then the
\emph{reverse} of $x$, written $R(x)$, is $x[n]x[n-1]\cdots x[1].$ A
word that equals its own reverse is called a \emph{palindrome}.

In this paper we are concerned with runs. A \emph{run} (or
\emph{maximal periodicity}) in a word $x$ is a factor $x[i\dd j]$
having minimum period $p$, length at least $2p$ and such that
neither $x[i-1\dd j]$ nor $x[i\dd j+1]$ is a factor with period $p$.
Runs are important because of their applications in data compression
and computational biology (see, for example, \cite{KK}). In recent
years a number of papers have appeared concerning the function
$\rho(n)$ which is the maximum number of runs that can occur in a
word of length $n$. In 2000 Kolpakov and Kucherov \cite{KK} showed
that $\rho(n)=O(n)$ but their method did not give any information
about the size of the implied constant. They conjectured that
$\rho(n)<n$ for all $n$ which has become known as the Runs
Conjecture.  In \cite{R2006} Rytter showed that $\rho(n)<5n$. This
bound was improved progressively in \cite{PSS} and \cite{CI2008} and
most recently by Crochemore, Ilie and Tinta \cite{CIT} to $1.029n$.
Their method is difficult and heavily computational. Giraud~\cite{G}
has produced weaker results using a much simpler technique. He also
showed that $\lim_{n\rightarrow \infty}\rho(n)/n$ exists. In the
other direction Franek \textit{et al.}~\cite{FSS} showed that this limit is
greater than 0.927, a result that was improved by Kusano \textit{et al.}~\cite{KMIBS} and Simpson~\cite{S} to 0.944.  We therefore have

\begin{equation*}
0.944 < \lim_{n\rightarrow \infty}\rho(n)/n <1.029.
\end{equation*}

These investigations have prompted authors to investigate a number
of associated problems.  Baturo and coauthors looked at runs in
Sturmian words \cite{BPR}. Puglisi and Simpson \cite{PS} gave
formulas for the expected number of runs in a word of length $n$.
This depends on the alphabet size, with binary alphabets giving the
highest values. Kusano and Shinohara \cite{KS} obtained similar
results for necklaces (words with their ends joined). Crochemore
\cite{CKRRW} and others have investigated runs whose length is at
least three times the period. Rather than counting the number of
runs one can consider the sum of the exponents of the runs. The word
$ababaabaa$ has runs  $(ab)^{5/2}$, $(aba)^{7/3}$ and two copies of
$a^2$, so it contains 4 runs with sum of exponents $53/6$. Let
$\epsilon(n)$ be the maximum sum of the exponents of runs in a word
of length $n$. It is known \cite{CKRRW} that for large $n$
\begin{equation*}
2.035n < \epsilon(n) <4.1n.
\end{equation*}

In this paper we introduce a new variation on this theme.  The
\emph{total run length} of a word is the sum of the lengths of the
runs in the word. The word given above contains runs $aa$ (twice),
$ababa$ and $abaabaa$ of lengths $2$, $5$, and $7$ so its total run
length is $16$. We write $TRL(w)$ for the total run length of a word
$w$ and $\tau(n)$ for $\max\{TRL(w):|w|=n\}$. In the next section we
give some minor results about \textit{total run length} (TRL for
short) and obtain a lower bound on $\tau(n)$.  An upper bound is
given in Section~\ref{S:Sec3} and formula for the expected TRL in
Section~\ref{S:Sec4}.  In the final section we discuss our results
and suggest some areas for further research.

\section{A lower bound on $\tau(n)$}

Table~1 (below) gives values of $\tau(n)$ for small
$n$ (under the assumption that these values are attained by binary
words) and examples of words that attain these values.

\begin{table}[htb!] \label{Table 1}
\begin{center}
\caption{Values of $\tau(n)$ assuming these values are attained by binary
words.} \label{tau}
    \begin{tabular}{llll}
      $n$ & $\tau(n)$ & $\tau(n)/n^2$ & Example\\
     \hline
        1 & 0 & 0 &$a$\\
2 & 2 & 0.5 & $aa$\\
3 & 3 & 0.333 & $aaa$\\
4 & 4 & 0.250 & $aaaa$\\
5 & 6 & 0.240 & $aabab$\\
6 & 10 & 0.278 & $aabaab$\\
7 & 12 & 0.245 & $aabaabb$\\
8 & 16 & 0.250 & $aabbaabb$\\
9 & 19 & 0.235 & $abaaabaab$\\
10 & 29 & 0.290 & $aababaabab$\\
11 & 32 & 0.264 & $abaababaaba$\\
12 & 37 & 0.257 & $abaababaabab$\\
13 & 42 & 0.249 & $ababbababbaba$\\
14 & 47 & 0.240 & $aaabaabaaabaab$\\
15 & 53 & 0.236 & $abaabababaababa$\\
16 & 60 & 0.234 & $aabaababaabaabab$\\
17 & 70 & 0.242 & $ababaabababaababa$\\
18 & 73 & 0.225 & $aababaabababaababa$\\
19 & 80 & 0.222 & $abaababaabaababaaba$\\
20 & 85 & 0.212 & $abaababaabaababaabab$\\
21 & 92 & 0.209 & $ababaababababaabababa$\\
22 & 99 & 0.205 & $aababaababaaababaababa$
    \end{tabular}
  \end{center}
\end{table}

We do not know whether binary words are best, though this seems likely.  The
same uncertainty exists for $\rho(n)$ and is discussed in
\cite{BGKNSS}. Note that if a binary word is optimal with respect to TRL, then so is its reverse and its \emph{complement} (formed by interchanging the letters $a$ and $b$).  In most cases, the
binary words attaining the values in Table~1 are unique up to reversal and complementation. Also note that, in many cases, binary words having maximum TRL are palindromes.

If we use an alphabet of size greater than $2$ we can construct words
of any length containing no runs (see Section 3.1.2 of \cite{Lo2})
and so having TRL equal to 0. For binary words the minimum values of
$TRL(w)$ for $w$ of length up to 5 are 0,0,0,2,2. For larger values
of $n$ we have the following.

\begin{thm} The minimum value of $TRL(w)$ for binary words of length $n$,
$n \ge 6$, is $n-4$, and is attained by the word $aba^{n-4}ba$.
\end{thm}
\begin{proof} Clearly the given word has TRL equal to $n-4$.  We
must show no binary word of length $n$ has a lower TRL.  Suppose
the word $w$ does.  Then it has $TRL(w) \le n-5$ and therefore there
are at least 5 letters in $w$ which do not belong to any run.
Consider the middle of these 5, and without loss of generality
suppose it's $a$. Its neighbours cannot equal $a$ as then it would
belong to a run. Therefore it is the central $a$ of some factor
$ab^{k_1}ab^{k_2}a$. If $k_1 \le k_2$ then we have a preficial run
$ab^{k_1}ab^{k_1}$ so the central $a$ does belong to a run,
contradicting the hypothesis.  If $k_1 > k_2$ an analogous argument
applies.

\end{proof}

\begin{thm} \label{lowertau} For $n>1$ we have $\tau(n) > n^2/8$.\end{thm}
\begin{proof} From Table 1 we see this holds up to $n=5$. For even $n$
greater than 5 let $u(k)=((ab)^ka)^2$. We find
\begin{eqnarray*}
&&n=|u(k)|=4k+2
\end{eqnarray*}
and
\begin{eqnarray*}
TRL(u(k))&=& 2k^2+8k+4\\
&=& (n^2+4n+12)/8
\end{eqnarray*} so $\tau(n) \ge (n^2+4n+12)/8$ and the theorem holds for even $n$. For odd $n$ note
that $\tau(n)>\tau(n-1)\ge (n^2+2n+9)/8$ so the bound also holds in this case too.

\end{proof}

\section{An upper bound for $\tau(n)$} \label{S:Sec3}

We first assemble some lemmas.

\begin{lem}\cite{FW} \label{lem 4.1} \textup{(The Periodicity
Lemma)} If $x$ is a word having two periods $p$ and $q$ and $|x| \ge
p+q-\gcd(p,q)$ then $x$ also has period $\gcd(p,q)$.\end{lem}

\begin{lem} \label{LA} \textup{(Lemma 8.1.1 of \cite{Lo2})} Let $\mathbf{a}$ be a word having two periods $p$ and
$q$ with $q<p$. Then the suffix and prefix  of length
$|\mathbf{a}|-q$ both have period $p-q$.\end{lem}

\begin{lem} \label{LC}\textup{(Lemma 8.1.2 of \cite{Lo2})}
Let  $\mathbf{a}$, $\mathbf{b}$ and $\mathbf{c}$ be words such that
$\mathbf{ab}$ and $\mathbf{bc}$ have period $p$ and $|\mathbf{b}|
\ge p.$ Then the word $\mathbf{abc}$ has period $p$.
\end{lem}

\begin{lem} \label{L:pp+1}
If $w$ is  a word for which $w[1..2p]$ has period $p$ and
$w[k+1..k+2p+2]$ has period $p+1$, where $0 \le k \le p$ then $w$
has the form:
\begin{equation}\label{pp+1}
w=Xx^{p-k}Xx^{p-k+1}Xx
\end{equation}\label{e5} where $x$ is a letter and $|X|=k$.
\end{lem}
\begin{proof} Note that $w[k+1..2p]$ has periods $p$ and $p+1$. By Lemma \ref{LA}
its prefix $w[k+1..p]$ (which is empty if $k=p$) has period $1$. Say
this is $x^{p-k}$. Then by the $p$ periodicity
$w[p+k+1..2p]=x^{p-k}$. By the $p+1$ periodicity
$w[k+2p+2]=w[p+k+1]=x$ and $w[2p+1]=w[p]=x$. Let $w[1..k]=X$ so that
$|X|=k$. Then by the $p$ periodicity $w[p+1..p+k]=X$ and then, by
the $p+1$ periodicity, $w[2p+2..2p+k+1]=X$. Assembling all this
gives \eqref{pp+1}.
\end{proof}

\noindent \textit{Remark.} It is clear that if $w[1..2p+2]$ has
period $p+1$ and $w[k+3..2p+2]$ has period $p$ then $w$ is the
reverse of the right hand side of \eqref{pp+1}.\\

\begin{thm} \label{simult} It is not possible for a letter to simultaneously belong to two
distinct runs with period $p$ and two distinct runs of period $p+1$.
\end{thm}
\begin{proof}
The proof is by contradiction. If there exists a counterexample to
the theorem then it has a prefix and a suffix  each of length $2p$
and each of which has a length $2p$ prefix with period $p$ and a
length $2p+2$ suffix with period $p+1$, or vice versa, and is such
that the four periodic factors have at least one letter in common.

Let $$ \alpha=xXx\underline{x^sXx^s}X$$ and
$$\beta=yYy\underline{y^tYy^t}Y$$ where
\begin{equation}\label{Case 3}
|X|+s=|Y|+t=p. \end{equation} By Lemma~\ref{L:pp+1} each of $\alpha$
and $\beta$ has a prefix of length $2p+2$ with period $p+1$ and a
suffix of length $2p$ and period $p$. The intersection of these two
squares is underlined. We write $R(\alpha)$ and $R(\beta)$ for the
reverses of $\alpha$ and $\beta$.  We consider four cases.\\

\noindent \textit{Case 1.} A word $w$ has prefix $\alpha$ and suffix
$R(\beta)$ with the underlined
factors having non-empty intersection.\\
\textit{Case 2.} A word $w$ has
prefix $R(\alpha)$ and suffix $\beta$ with the underlined factors having non-empty intersection.\\
\textit{Case 3.} A word $w$ has
prefix $\alpha$ and suffix $\beta$ with the underlined factors having non-empty intersection.\\
\textit{Case 4.} A word $w$ has
prefix $R(\alpha)$ and suffix $R(\beta)$ with the underlined factors having non-empty intersection.\\

If the statement of the Theorem is incorrect then a word belonging
to one of these cases must exist with the stated periods being
minimal and the four squares belonging to four different runs. We
will show that in each case this cannot occur.\\

\noindent \textit{Case 1.} We have
\begin{equation*}
\alpha=xXx\underline{x^sXx^s}X \quad \mbox{and} \quad  R(\beta)=Y\underline{y^tYy^t}yYy.
\end{equation*} Let $d$ be the length of the intersection of these
two words.  The intersection must have length less than $p$ else, by
Lemma \ref{LC}, the two period $p$ squares would belong to the same
run of period $p$. We must also have $d>|X|+|Y|$ else the underlined
factors would not intersect. Recall that $|x^sX|=|y^tY|=p$ so the
intersection is a suffix $x^iX$ of $x^sX$ and a prefix $Yy^j$ of
$Yy^t$. Thus
$$|x^i|+|X|>|X|+|Y|$$ so $i>|Y|$ which implies that $x=y$ and that of $X$ and $Y$ is a power of
$x$. It follows that $w$ is a power of $x$ and the four squares belong to a single run.\\

\noindent \textit{Case 2.} We have
\begin{equation*}
R(\alpha)=X\underline{x^sXx^s}xXx \quad \mbox{and} \quad \beta=yYy\underline{y^tYy^t}Y.
\end{equation*} Let $d$ be the length of the intersection of these
two words. Now $R(\alpha)$ and $\beta$ have, respectively, a suffix
with period $p+1$ and a prefix with period $p+1$. As in Case 1 we
must have $d<p+1$.  In order that the underlined factors intersect
we must also have $d>4+|X|+|Y|$.  The intersection is thus
$$x^iXx=yYy^j$$ where $i+1+|X|>|X|+|Y|+4$ implying $i>|Y|+3$. As in
Case 1 this implies that $x=y$ and that $X$ and $Y$ are powers of
$x$ as is the whole word $w$.\\

\noindent \textit{Case 3.} We have
\begin{equation*}
\alpha=xXx\underline{x^sXx^s}X \quad \mbox{and} \quad \beta=yYy\underline{y^tYy^t}Y.
\end{equation*}
This case is more complicated than the others.  Let us add another
$y$ to the right hand end of $\beta$. Set $\beta'=\beta y$ and
$Y'=Yy$. Then
$$\beta'=yY'y\underline{y^{t-1}Y'y^{t-1}}Y'.$$ This has the same form as
$\beta$ but its underlined factor is one letter shorter and begins
one position further to the right.  Suppose we have a word with
prefix $\alpha$ and suffix $\beta$ in which the underlined factors
intersect. By iterating the construction just described we can
arrange that the two underlined factors have an intersection of
length one. This will be the final $x$ in the underlined factor of
$\alpha$ and the initial $y$ in the underlined factor of $\beta$.
Thus $x=y$. We suppose, without loss of generality, that this is the
case with our word.

$X$ and $Y$ may have prefixes or suffixes which are powers of $x$.
We set
\begin{eqnarray*}
X=x^aUx^b\\
Y=x^cVx^d
\end{eqnarray*}
for non-negative integers $a$, $b$, $c$ and $d$, where $U$ and $V$
neither begin nor end with $x$. Equations (\ref{Case 3}) now become
\begin{equation}\ \label{e7}
|U|+a+b+s=|V|+c+d+t=p.
\end{equation} and we have
\begin{eqnarray*}
&&\alpha=x^{a+1}Ux^{b+1}\underline{x^{a+s}Ux^{b+s}}\;x^aUx^b\\
&&\beta=x^{c+1}Vx^{d+1}\underline{x^{c+t}Vx^{d+t}}\;x^cVx^d.
\end{eqnarray*} By our assumption the last $x$ in the underlined
section of $\alpha$ coincides with the first $x$ in the underlined
section of $\beta$. This means that $x^{c+1}Vx^{d+1}$ is a suffix of
$x^{a+1}Ux^{b+1}\underline{x^{a+s}Ux^{b+s-1}}$ so that $U=V$ and
$b+s-1=d+1$. It also means that $x^aUx^b$ is a prefix of
$\underline{x^{c+t-1}Vx{d+t}}x^cVx^{d}$ so that $a=c+t-1$. Together
this gives $a+b+s-1=c+d+t$, which contradicts (\ref{e7}).\\

\noindent \textit{Case 4.} This is just the reverse of Case 3 and
need not be separately considered.  The proof is complete.

\end{proof}


\pagebreak

\begin{thm} \label{uppertau} For all $n$ we have $\tau(n)<47n^2/72+2n.$
\end{thm}
\begin{proof} Periods of runs in a
word of length $n$ must be less than or equal to $n/2$.

Consider runs with periods in $\{2q-1,2q\}$ for $1\le q \le
 \lfloor n/6 \rfloor$. By Theorem~\ref{simult} no letter can belong to
more than three such runs so the contribution to the TRL is at most
$3n$ for each such pair, and the contribution from all such pairs is
at most $3n\lfloor n/6 \rfloor$.

Now consider runs with periods in $\{2q-1,2q\}$ for $\lfloor n/6
\rfloor +1 \le q \le \lceil n/4 \rceil$. The upper bound here
ensures that the maximum value of $2q$ is at least equal to $\lfloor
n/2 \rfloor$.  For some values of $n$ we will be counting more runs
than we need. The number of pairs $\{2q-1,2q\}$ is $\lceil n/4
\rceil-\lfloor n/6 \rfloor$.

We first show that there can be at most one run in a word of length
$n$ for each period in the set under consideration.  Let $p$ be such
a period.  Then $p\ge 2(\lfloor n/6 \rfloor+1)>n/3$. If we had two
runs with period $p$ their intersection would have length at least
$4p-n$ which is greater than $p$.  This is impossible by  Lemma
\ref{LC}. So we have at most one run for each period $p$.  Suppose
there is a run of length $x$ with period $2q-1$ and a run of length
$y$ with period $2q$.  These have intersection of length at least
$x+y-n$. By Lemma~\ref{lem 4.1}, this must be less than
\[
2q+2q-1-\gcd(2q-1,2q)=4q-2
\]
else the runs will collapse into a
single run with period 1.  So $x+y \le n+4q-3$.  The contribution
from all such pairs to the TRL is at most
\begin{equation*}
\sum_{q=\lfloor n/6 \rfloor +1}^{\lceil n/4 \rceil} n+4q-3=(\lceil
n/4\rceil - \lfloor n/6 \rfloor)(n-3+2(\lceil n/4 \rceil+\lfloor n/6
\rfloor+1)).
\end{equation*}  Adding this to the bound for the shorter periods we
see that the TRL is less than
\begin{equation*}(\lceil n/4\rceil - \lfloor n/6
\rfloor)(n-3+2(\lceil n/4 \rceil+\lfloor n/6 \rfloor+1))+3n\lfloor
n/6 \rfloor.
\end{equation*}
We can show that this is less than  the bound in the theorem by
considering values of $n$ in each residue class modulo 12.

\end{proof}

\section{The expected value of TRL} \label{S:Sec4}

\begin{thm} \label{T:expected} The expected $TRL$ for a word of length $n$ on an
alphabet of size $\alpha$ is
{\small
\begin{equation}\label{4.1}
\frac{(\alpha -1)^2}{\alpha^2}\sum_{p=1}^{\lfloor (n-2)/2 \rfloor}
P(p) \sum_{i=1}^{n-2p-1}\sum_{k=2p}^{n-i-1}
k\alpha^{-k}+2\frac{\alpha -1}{\alpha}\sum_{p=1}^{\lfloor (n-1)/2
\rfloor} P(p)
\sum_{k=2p}^{n-1} k\alpha^{-k} + \frac{n}{\alpha^n}\sum_{p=1}^{\lfloor n/2 \rfloor}
P(p),
\end{equation}}where $P(p)=\sum_{d|p}\alpha^d \mu(p/d)$ is the
number of length $p$ primitive words on an alphabet of size $\alpha$
(see \cite[Eq.~1.3.7]{Lo1}) and $\mu$ is the M\"{o}bius
function.
\end{thm}
\begin{proof} We count the sum of the TRLs of all words of length
$n$ on an alphabet of size $\alpha$. We first sum the TRLs of those
runs which are neither prefixes nor suffixes.

Consider runs of the form $x[i+1..i+k]$, where $1\le i$ and $i+k<n$,
which have period $p$. For such runs $x[1..i-1]$ can be any word, so
there are $\alpha^{i-1}$ possibilities for this factor. The letter
$x[i]$ must be chosen so that the run does not extend to the left of
$x[i+1]$.  There are $\alpha-1$ such choices.  The factor
$x[i+1..i+p]$ is the generator of the run and can be any primitive
word of length $p$, for which there are $P(p)$ choices. The rest of
the run is then determined by its periodicity.  The letter
$x[i+k+1]$ is chosen in one of $\alpha -1 $ ways to avoid the run
extending to the right.  This leave the final factor $x[i+k+2..n]$
which can be chosen in $\alpha^{n-i-k-1}$ ways. The number of words
having a run of the required form is therefore
\begin{equation*}
\alpha^{i-1}(\alpha-1)P(p)(\alpha-1)\alpha^{n-i-k-1}=(\alpha-1)^2\alpha^{n-k-2}P(p).
\end{equation*}
The variable $i$ can take any value from 1 to $n-2p-1$ and, for each
such $i$, $k$ can take the values $2p$ to $n-i-1$. The length of the
run is $k$ so the sum of total run lengths of all runs which are not
suffixes are prefixes, which have period $p$, in all words of length
$n$ is:
\begin{eqnarray*}
&&P(p)\sum_{i=1}^{n-2p-1}\sum_{k=2p}^{n-i-1}(\alpha-1)^2\alpha^{n-k-2}k\\
&=&(\alpha-1)^2\alpha^{n-2}P(p)\sum_{i=1}^{n-2p-1}\sum_{k=2p}^{n-i-1}\alpha^{-k}k.
\end{eqnarray*}
Now consider those runs which are prefixes of $x$ but not suffixes
(that is, their length is less than $n$). Say $x[1..k]$ is such a
run with period $p$.  We have $P(p)$ choices for $x[1..p]$,
$x[p+1..k]$, $x[k+1]$ can be chosen is $\alpha -1
$ ways and the rest of the word in $\alpha^{n-k-1}$ ways. The run
length $k$ can take any value from $2p$ to $n-1$. The sum of the
total run lengths of all prefix runs with period $p$, in all words
of length $n$ is:
\begin{eqnarray*}
P(p)\sum_{k=2p}^{n-1}(\alpha-1)\alpha^{n-k-1}k
=(\alpha-1)\alpha^{n-1}P(p)\sum_{k=2p}^{n-1}\alpha^{-k}k.
\end{eqnarray*}
By symmetry this is also the total for runs which are suffixes but
not prefixes.  Finally the number of runs which cover the whole word
is just $P(p)$ and these all have length $n$.  The sum of the total
run length of all runs with period $p$ is therefore:

{\small
\begin{equation*}
(\alpha-1)^2\alpha^{n-2}P(p)\sum_{i=2}^{n-2p-1}\sum_{k=2p}^{n-i-1}\alpha^{-k}k+
2(\alpha-1)\alpha^{n-1}P(p)\sum_{k=2p}^{n-1}\alpha^{-k}k+P(p)n.
\end{equation*}
}

A complication arises here because the maximum period $p$ depends on
which of the four cases we are considering.  It is not hard to see
that if the run is neither a prefix nor a suffix then its period is
at most $\lfloor (n-2)/2 \rfloor$, if it is a prefix but not a
suffix, or vice versa,  then its period $p$ is at most $\lfloor
(n-1)/2 \rfloor$ and when it is both a prefix and a suffix, $p$ is at most $\lfloor n/2
\rfloor$. Allowing for these different bounds, summing over $p$ and
dividing through by $\alpha^n$ (the number of words of length $n$)
gives the required formula.

\end{proof}

Let us say that the \emph{TRL-density} of a word $x$ is $TRL(x)/|x|$.

\begin{cor} \label{cor 4.2} As $n$ tends to infinity the expected density of a
word on alphabet size $\alpha$ tends to
\begin{equation*}
\lim_{n\rightarrow\infty}\sum_{p=1}^n
P(p)\frac{2p(\alpha-1)+1}{\alpha^{2p+1}}
\end{equation*} where $P(p)$ is as defined in Theorem~$\ref{T:expected}$.
\end{cor}
\begin{proof} We write $S_1(n)$, $S_2(n)$ and  $S_3(n)$ for the
three terms in (\ref{4.1}), that is,
\begin{align}
\label{S1}&S_1(n)=\frac{(\alpha -1)^2}{\alpha^2}\sum_{p=1}^{\lfloor
(n-2)/2 \rfloor} P(p) \sum_{i=1}^{n-2p-1}\sum_{k=2p}^{n-i-1}
k\alpha^{-k},\\
\nonumber &S_2(n)=2\frac{\alpha -1}{\alpha}\sum_{p=1}^{\lfloor
(n-1)/2
\rfloor} P(p) \sum_{k=2p}^{n-1} k\alpha^{-k},\\
\nonumber &S_3(n)=\frac{n}{\alpha^n}\sum_{p=1}^{\lfloor n/2 \rfloor}
P(p).
\end{align}
We write $S(n)$ for $S_1(n)+S_2(n)+S_3(n)$.  We will obtain $\lim_{n
\rightarrow \infty}S(n+1)-S(n)$ and show that this is a finite
constant depending only on $\alpha$.  It follows that this limit
equals $\lim_{n \rightarrow \infty}S(n)/n$ which is the required
expected density. It is easy to show that $\lim_{n \rightarrow
\infty}S_3(n)$ equals 0.  This is not surprising since $S_3(n)$
counts the contribution to $S(n)$ from words which are themselves
runs.  Such words are rare among the $\alpha^n$ words of length $n$.
It follows that
\begin{equation}\label{4.2.1}
\lim_{n \rightarrow \infty}S_3(n+1)-S_3(n)=0.
\end{equation}
Now consider the term $S_2(n)$.  This equals
\begin{align*}
 \frac{2}{\alpha -1}\sum_{p=1}^{\lfloor (n-1)/2
\rfloor}P(p)
\left\{\frac{2p(\alpha-1)+1}{\alpha^{2p}}-\frac{n(\alpha-1)+1}{\alpha^{n}}\right\}.
\end{align*} As $n$ goes to infinity the sum of the second term in the
parentheses goes to 0 so we have
\begin{equation*}
\lim_{n \rightarrow \infty}S_2(n)=\frac{2}{\alpha-1}\sum_{p=1}^{\lfloor (n-1)/2 \rfloor}P(p) \frac{2p(\alpha-1)+1}
{\alpha^{2p}}.
\end{equation*} Noting that $1 \le P(p) \le \alpha^p$ we see that this
limit exists and is finite.  For $\alpha=2$ it equals 10. It follows
that
\begin{equation}\label{4.2.2}
\lim_{n \rightarrow \infty}S_2(n+1)-S_2(n)=0.
\end{equation}
Next consider $S_1(n)$. A change in order of summation gives
\begin{equation*}
S_1(n)=\frac{(\alpha-1)^2}{\alpha^2}
\sum_{i=1}^{n-3}\sum_{k=2}^{n-1-i}k\alpha^{-k}\sum_{p=1}^{\lfloor
k/2 \rfloor}P(p).
\end{equation*} Then $S_1(n+1)-S_1(n)$ equals
\begin{equation*}\label{4.2.3}
\frac{(\alpha-1)^2}{\alpha^2}
\left\{\sum_{k=2}^2k\alpha^{-k}\sum_{p=1}^{\lfloor k/2
\rfloor}P(p)+\sum_{i=1}^{n-3}(n-i)\alpha^{-n+i}\sum_{p=1}^{\lfloor
(n-i)/2 \rfloor}P(p)\right\}.
\end{equation*} The first term in the parentheses is the $i=n-2$
term in the first sum in (\ref{S1}).  The second term corresponds to
the terms with $k=n-i$.  Since $P(1)=\alpha$ for any $\alpha$ this
becomes, after changing the index of summation to $j=n-i$,
\begin{equation*}
\frac{(\alpha-1)^2}{\alpha^2}
\left\{\frac{2}{\alpha}+\sum_{j=3}^{n-1}j\alpha^{-j}\sum_{p=1}^{\lfloor
j/2 \rfloor}P(p)\right\}.
\end{equation*} Changing the order of summation again gives
\begin{equation*}
\frac{(\alpha-1)^2}{\alpha^2} \left\{\frac{2}{\alpha}+\sum_{p=1}^{\lfloor
(n-1)/2 \rfloor}P(p)\sum_{j=\max(3,2p)}^{n-1}j\alpha^{-j}\right\}.
\end{equation*}
To simplify matters we start the second sum at $j=2p$.  This means
we are including an unwanted term corresponding to $j=2$, $p=1$.
This is $2P(1)/\alpha^2=2/\alpha$, which equals the first term in
the parentheses.  We thus have
\begin{align*}
S_1(n+1)-S_1(n)&=\frac{(\alpha-1)^2}{\alpha^2} \sum_{p=1}^{\lfloor
(n-1)/2
\rfloor}P(p)\sum_{j=2p}^{n-1}j\alpha^{-j}\\
&= \frac{1}{\alpha} \sum_{p=1}^{\lfloor (n-1)/2
\rfloor}P(p)\left\{\frac{2p(\alpha-1)+1}{\alpha^{2p}}-\frac{n(\alpha-1)+1}{\alpha^{n}}\right\}.\\
\end{align*} We now take the limit as $n$ goes to infinity. The
second term makes no contribution to this limit since it's dominated
by $\alpha^{-n}$.  So we have
\begin{equation*}
\lim_{n \rightarrow \infty}S_3(n+1)-S_3(n)=\lim_{n \rightarrow
\infty}\sum_{p=1}^nP(p)\frac{2p(\alpha-1)+1}{\alpha^{2p+1}}.
\end{equation*} Summing this with (\ref{4.2.1}) and (\ref{4.2.2})
completes the proof.

\end{proof}

Some values of expected TRL-density are given in Table 2 below, along with
corresponding results for the number of runs and the sum of
exponents of runs.

\begin{table}[htb!] \label{Table 2}
\begin{center}
\caption{The columns show the expected number of runs per units
length of a word, the expected sum of exponents per units length and
the expected TRL per unit length for various alphabet sizes. The
values come from \cite{PS}, \cite{KMIS} and Corollary \ref{cor 4.2}
respectively.}

    \begin{tabular}{clll}
      Alphabet size & Runs & Exponents & TRL\\
     \hline
    2&0.4116&1.1310&1.9775\\
    3&0.3049&0.7382&1.0290\\
    5&0.1933&0.4304&0.5208\\
    10&0.0991&0.2087&0.2296
    \end{tabular}
  \end{center}
\end{table}
\section{Discussion}
Theorems \ref{lowertau} and \ref{uppertau} show that, for all $n$,
\begin{equation*}
\frac{1}{8}<\frac{\tau(n)}{n^2}<\frac{47}{72}+\frac{2}{n}.
\end{equation*} Both these bounds might be improved.  Lower bounds
for $\rho(n)$ and $\epsilon(n)$ were obtained by constructing words
which were rich in the appropriate way. The word $u(n)$ of Theorem~\ref{lowertau} is comparatively simple. One could look for something
better using the techniques of \cite{MKBS} or some combinatorial
heuristic such as simulated annealing or genetic algorithms.

The upper bound is probably far from best when $n$ is large, though
from Table 1 we suspect that the maximum value of
$\tau(n)/n^2$ occurs when $n=2$ and it would seem that
$\lim_{n \rightarrow \infty} \tau(n)/n^2$ exists, but we have not been able
to prove it. Giraud's method \cite{G} for showing the existence of
$\lim_{n \rightarrow \infty} \rho(n)/n$ does not seem applicable to
our situation.  His method also showed that the limit is the
supremum of the function. In our case it may be the infimum of
$\{\tau(n)/n^2:n>1\}$.  Extending Table~1 might give insight into
these questions.

\medskip
\noindent \textbf{Acknowledgements:} We thank the referee for careful reading and
helpful comments.

\end{document}